\documentclass{article}

 \usepackage{graphicx}
 \usepackage{epstopdf}
\usepackage{amsmath, amssymb, amsthm, enumerate}
 \usepackage{subfig}
 \usepackage{float}
\usepackage{amsmath,amssymb,amsthm,enumerate,mathrsfs}
\newtheorem{theorem}{Theorem}[section]
\newtheorem{corollary}[theorem]{Corollary}
\newtheorem{lemma}[theorem]{Lemma}
\newtheorem{proposition}[theorem]{Proposition}
\theoremstyle{definition}
\newtheorem{definition}[theorem]{Definition}

\newtheorem{assumption}[theorem]{Assumption}

\newtheorem{remark}[theorem]{Remark}
\newtheorem{algorithm}[theorem]{Algorithm}
\usepackage{color}

\numberwithin{equation}{section}

\begin{document}

\renewcommand{\baselinestretch}{1.2}

\title {Inexact Newton Method for M-Tensor Equations\thanks{Supported by the NSF of China grant 11771157
and the NSF of Guangdong Province grant No.2020B1515310013.}}
\author{Dong-Hui Li
\thanks{School of Mathematical Sciences, South China Normal University, Guangzhou, 510631, China,
lidonghui@m.scnu.edu.cn}
\ \ Hong-Bo Guan\thanks{School of Mathematical Sciences, South China Normal University, Guangzhou, 510631, China,
hongbo guan@m.scnu.edu.cn.\,\,
School of Mathematics, Physics and Energy Engineering, Hunan Institute of Technology,
Hengyang, 421002, China.} 
\ \ Jie-Feng Xu\thanks{School of Mathematical Sciences, South China Normal University, Guangzhou, 510631, China,
2018021699@m.scnu.edu.cn.}
\\
%School of Mathematical Sciences, \\
%South China Normal University, Guangzhou, 510631, China. \\
%Email: lidonghui@m.scnu.edu.cn
}
\maketitle

\begin{abstract}
  We first investigate properties of M-tensor equations. In particular, we show that if the constant term
of the equation is nonnegative, then finding a nonnegative solution of the equation can be done by finding
a positive solution of a lower dimensional M-tensor equation.
We then propose an inexact Newton method to find a positive solution to the lower dimensional
equation and establish its global convergence. We also show that the convergence rate
of the method is quadratic. At last, we do numerical experiments to test the proposed Newton method.
The results show that the proposed  Newton  method has a very good numerical performance.
\end{abstract}


{\bf AMS} 65H10, 65K10, 90C30

\section{Introduction}
Tensor equation is a special system of nonlinear equations. It is also called multilinear equation.
Tensor equation can be expressed as
\begin{equation}\label{eqn:M-teq}
F(x)={\cal A}x^{m-1} - b=0,
\end{equation}
where $x,b\in \mathbb{R}^n$ and ${\cal A}$ is an ${m}$th-order ${n}$-dimensional tensor that takes the form
\[
{\cal A} = (a_{i_1i_2\ldots i_m}), \quad  a_{i_1i_2\ldots i_m}\in \mathbb{R}, \quad 1\le i_1, i_2, \cdots, i_m \le n,
\]
and ${{\cal A}x^{m-1}}\in \mathbb{R}^n$  with elements
\[
({\cal A}x^{m-1})_i = \sum_{i_2, \ldots, i_m}a_{ii_2\ldots i_m}x_{i_2}\cdots x_{i_m},\quad i=1,2,\ldots,n.
\]
The notation  ${\cal A}x^m$  will denote the homogenous polynomial of degree $m$, i.e.,
\[
{\cal A}x^m=x^T{\cal A}x^{m-1}= \sum_{i_1,\ldots, i_m} a_{i_1\ldots i_m}x_{i_1}\cdots x_{i_m}.
\]

For convenience of presentation, we introduce some concepts and notations,
 which will be used throughout the paper.
We denote the set of all ${m}$th-order ${n}$-dimensional tensors by ${\cal T}(m,n)$.
We first introduce the concepts of Z-matrix and M-matrix.
\begin{definition}{\em\cite{Berman-Plemmons-94}}\label{marti-ZM}
A matrix $A$ is called a Z-matrix if all its off-diagonal entries are non-positive. It is
apparent that a Z-matrix $A$ can be written as
\[
A=sI-B,
\]
where $B$ is a nonnegative matrix ($B\geq 0$) and $s>0$; When $s\geq \rho(B)$,
we call $A$ is an M-matrix; And further when $s>\rho(B)$, we call $A$ as a nonsingular M-matrix.
\end{definition}
The concept of M-tensor is an extension of the definition of M-matrix. Now we introduce the definition of M-tensor and other structure tensors that will be involved in this paper.

\begin{definition}
{\rm\cite{Chang-Pearson-Zhang-11,Ding-Qi-Wei-13,Ding-Wei-16,Lim-05,Qi05,Qi-Chen-Chen-18,Qi-Luo-17,Zhang-Qi-Zhou-2014}}
Let ${\cal A}\in {\cal T}(m,n)$.
\begin{itemize}
\item   ${\cal A}$ is called a non-negative tensor, denoted by ${\cal A}\ge 0$,
if all its elements are non-negative, i.e., $a_{i_1i_2\ldots i_m}\ge 0$, $\forall i_1,\ldots, i_m\in [n]$,
where $[n]=\{1,2\cdots,n\}$.

\item  ${\cal A}$ is called a symmetric tensor,
if its elements $a_{i_1i_2\ldots i_m}$ are invariant under any permutation of their indices.
In particular, for every index $i\in [n]$, if an (m-1)th order n-dimensional square tensor
${\cal A}_i:=(a_{ii_2\ldots i_m})_{1\leq i_2,\ldots,i_m\leq n}$ is symmetric, then ${\cal A}$ is called semi-symmetric tensor with respect to the indices $\{i_2,\ldots,i_m\}$.
The set of all ${m}$th-order ${n}$-dimensional symmetric tensors is denoted by ${\cal ST}(m,n)$.

\item  ${\cal A}$ is called the identity tensor, denoted by ${\cal I}$,
if its diagonal elements are all ones and other elements are zeros,
i.e., all $a_{i_1i_2\ldots i_m}=0$ except $a_{ii\ldots i}=1$, $\forall i\in [n]$.

\item If a real number $\lambda$ and a nonzero real vector $x\in \mathbb{R}^n$ satisfy
$${\cal A}x^{m-1}=\lambda x^{[m-1]},$$
    then $\lambda$ is called an H-eigenvalue of ${\cal A}$ and $x$ is called an H-eigenvector of ${\cal A}$
    associated with $\lambda$.

\item  ${\cal A}$ is called an M-tensor, if it can be written as
\begin{equation}\label{M-tensor}
{\cal A}=s{\cal I}-{\cal B},\quad {\cal B}\ge 0,\; s\ge \rho ({\cal B}),
\end{equation}
where $\rho ({\cal B})$ is the spectral radius of tensor ${\cal B}$, that is
\[
\rho(\cal B)= \max \left\{\left| \lambda \right|: \lambda \mbox{ is an eigenvalue of } \cal{B}\right \}.
\]
 If $s> \rho ({\cal B})$,
then ${\cal A}$ is called a strong or nonsingular M-tensor.

\item ${\cal A}$ is called a lower triangular tensor, if its possibly nonzero elements are
$a_{i_1i_2\ldots i_m}$
 with $i_1=1,2,\ldots,n$ and $i_2,\ldots,i_m\leq i_1$ and all other elements of ${\cal A}$ are zeros.
 ${\cal A}$ is
    called a strictly lower triangular tensor, if its possibly nonzero elements are
    $a_{i_1i_2\ldots i_m}$
    with $i_1=1,2,\ldots,n$ and $i_2,\ldots,i_m< i_1$ and all other elements of ${\cal A}$ are zeros.
\item ${\cal A}$ is called reducible it there is an index set
    $I\subset [n]$ such that the elements of $\cal A$ satisfy
\[
a_{i_1i_2\ldots i_m} = 0,\quad \forall i_1 \in I, \,
    \forall i_2,\ldots, i_m \notin I.
\]
If ${\cal A}$ is not reducible, then we call ${\cal A}$ irreducible.
\end{itemize}
\end{definition}

In the case ${\cal A}\in {\cal ST}(m,n)$, the derivative of the homogeneous polynomial ${\cal A}x^m$
can be expressed as  $\nabla ({\cal A}x^m)=m{\cal A}x^{m-1}$.

In the definition of reducible tensor, the index set $I\subset [n]$ can be arbitrary.
In our paper, we will need some special reducible tensor where  the index set $I$
is contained in some specified set. For the sake of convenience, we make a slight extension to
the definition of reducible tensors.
\begin{definition}
Tensor ${\cal A}\in {\cal T}(m,n)$ is called reducible respect to
    $I\subset [n]$ if its  elements satisfies
\[
a_{i_1i_2\ldots i_m} = 0,\quad \forall i_1 \in I, \,
    \forall i_2,\ldots, i_m \notin I.
\]
\end{definition}

It is easy to see that tensor ${\cal A}\in {\cal T}(m,n)$ is reducible if and only if there is
an index $I\subset [n]$ such that it is reducible respect to $I$.

We call index pair $(I,I_c)$ a partition to $[n]$ if $I, I_c\subset [n]$ and $I\cup I_n=[n]$.

For $x,y\in \mathbb{R}^n$ and $\alpha \in \mathbb{R}$,
the notations $x\circ y$ and $x^{[\alpha]}$ are vectors in $\mathbb {R}^n$
defined by
\[
x\circ y=(x_1y_1,\cdots,x_ny_n)^T
\]
and
\[
x^{[\alpha]}=(x_1^{\alpha},\ldots, x_n^\alpha)^T
\]
respectively.

We use $\mathbb {R}_+^n$ and $\mathbb {R}_{++}^n$ to denote the sets of all nonnegative vectors and positive
vectors in $\mathbb {R}^n$. That is,
\[
\mathbb {R}_+^n=\{x\in \mathbb {R}^n\;|\; x\ge 0\}\quad
\mbox{and}\quad \mathbb {R}_{++}^n=\{x\in \mathbb {R}^n\;|\; x> 0\}.
\]

If $\cal A$ is an M-tensor, we call the tensor equation an M-tensor equation and abbreviate it as M-Teq.

The following theorem comes from \cite{Berman-Plemmons-94,Ding-Wei-16, Gowda-Luo-Qi-Xiu-15}.
\begin{theorem}\label{th:1} Let ${\cal A}\in {\cal ST}(m,n)$.
\begin{itemize}
\item {\rm(\cite{Ding-Wei-16})} If ${\cal A}$ is a strong M-tensor and $b\in \mathbb {R}_{++}^n$, then the
M-Teq {\rm(\ref{eqn:M-teq})} has a unique positive solution.
\item {\rm(\cite{Gowda-Luo-Qi-Xiu-15})} If ${\cal A}$ is a strong M-tensor and $b\in \mathbb {R}_{+}^n$, then the
M-Teq {\rm(\ref{eqn:M-teq})} has a nonnegative solution.
\item {\rm(\cite{Berman-Plemmons-94})} For a Z-matrix $A\in \mathbb{R}^{n\times n}$, the following statements are equivalent.
\begin{itemize}
\item [{\rm(i)}] A is a nonsingular M-matrix.
\item [{\rm(ii)}] $Av\in \mathbb{R}_{++}^n$ for some vector $v\in \mathbb{R}_{++}^n$.
\item [{\rm(iii)}] All the principal minors of A are positive.
\end{itemize}
\end{itemize}
\end{theorem}

Tensor equation is also called multilinear equation. It appears in many practical
fields including data mining and numerical partial differential
equations \cite{4,Ding-Wei-16,8,Fan-Zhang-Chu-Wei-17,Kressner-Tobler-10,Li-Xie-Xu-17,Li-Ng-15,Xie-Jin-Wei-2017}.
The study in numerical methods for solving tensor equations has begun only a few years ago.
Most existing methods focus on solving the M-Teq under the restriction $b\in \mathbb {R}_{++}^n$ or
$b\in \mathbb {R}_{+}^n$.
Such as the iterative methods in \cite{Ding-Wei-16},
the homotopy method in \cite{Han-17},
the tensor splitting method in \cite{Liu-Li-Vong-18},
the Newton-type method in \cite{He-Ling-Qi-Zhou-18},
the continuous time neural network method in \cite{Wang-Che-Wei-19},
the preconditioned tensor splitting method in \cite{L-L-V-2018},
the preconditioned SOR method in \cite{Liu-Li-Vong-2020},
the preconditioned Jacobi type method in \cite{Zhang-Liu-Chen-2020},
the nonnegativity preserving algorithm in \cite{Bai-He-Ling-Zhou-18}.
There are also a few methods that can solve M-Teq (\ref{eqn:M-teq})  without restriction $b\in \mathbb {R}_{++}^n$
or
that $\cal A$ is an M tensor.
Those
methods include the splitting method by Li, Guan and Wang \cite{Li-Guan-Wang-18}, and
Li, Xie and Xu \cite{Li-Xie-Xu-17}, the  alternating projection method by
Li, Dai and Gao \cite{Dai-2019}, the alternating iterative methods by Liang, Zheng and Zhao \cite{Liang-Zheng-Zhao-2019} etc..
Related works can also be found in
\cite{B-H-C-2020,4,Li-Ng-15,Lv-Ma-18,Wang-Che-Wei-2020,Wang-Che-Wei-2019,Xie-Jin-Wei-2017a,Xie-Jin-Wei-2017,Huang-18,Yan-Ling-18}.

Newton's method is a well-known efficient method for solving nonlinear equations. An attractive property of
the method is its quadratic convergence rate.
However, in many cases, the standard
Newton method may fail to work or loss its quadratic convergence property when applied to solve
tensor equation (\ref{eqn:M-teq}). We refer to
\cite{Li-Guan-Wang-18} for details.

Recently, He, Ling, Qi and Zhou \cite{He-Ling-Qi-Zhou-18} proposed a Newton type method
for solving the M-Teq (\ref{eqn:M-teq}) with $b\in \mathbb {R}_{++}^n$. Unlike the standard
Newton method for solving nonlinear equations, by utilizing  the special structure of the equation (\ref{eqn:M-teq}),
the authors transformed the equation into an equivalent form through a variable transformation $y=x^{[m]}$.
Starting from some positive initial point, the method generates a sequence of positive iterates.
An attractive property of the method is that the Jacobian matrices of the equation at the iterates are
nonsingular. As a result, the method is well defined and retains the
global and quadratic convergence. The reported numerical results
in \cite{He-Ling-Qi-Zhou-18} confirmed the quadratic convergence property of that method.

It should be pointed out that the positivity of $b$ plays an important role in the Newton method
by He, Ling, Qi and Zhou \cite{He-Ling-Qi-Zhou-18}. It is not known if the method in \cite{He-Ling-Qi-Zhou-18}
is still well defined and reserves  quadratic convergence property  if there is some $i$ satisfying $b_i=0$.
The purpose of this paper is to develop a Newton method to find the a nonnegative solution of
the equation (\ref{eqn:M-teq}) with $b\in \mathbb {R}_{+}^n$. Our idea is similar to but different from that of
the method in  \cite{He-Ling-Qi-Zhou-18}. Specifically, we will reformulate the equation
via the variable transformation $y=x^{[m-1]}$. Such an idea comes from the following observation.
Consider a  vary special tensor equation
\[
Ax^{[m-1]}-b=0,
\]
corresponding to the tensor equation (\ref{eqn:M-teq})
where the only nonzero elements of $\cal A$ are
$a_{ij\ldots j}=a_{ij}$, $i,j=1,2,\ldots,n$. For that special equation, the tensor equation is
equivalent to the system of linear equation
$Ay-b=0$ with $y=x^{[m-1]}$. As a result, the corresponding Newton method terminates
at a solution of the equation within one iteration.
Another difference between our method and the method in  \cite{He-Ling-Qi-Zhou-18}
is that we will consider the equation (\ref{eqn:M-teq}) with $b\in \mathbb {R}_{+}^n$.
The case where $b$ has zero elements cause the problem be much more difficult.
Existing techniques that deals with equation (\ref{eqn:M-teq}) with $b\in \mathbb {R}_{++}^n$ are no longer available.
To overcome that difficult, we will propose a criterion that can identify the zero elements
in a nonnegative solution of the M-tensor equation. From computational view point, the
criterion is easy to implement. By the use of that criterion, we can
get a nonnegative solution of the M-tensor equation (\ref{eqn:M-teq}) by finding a positive solution to a
lower dimensional M-tensor equation with nonnegative constant term.

Based on that criterion, we propose a Newton method for finding a positive solution of the
M-Teq  with $b\in \mathbb {R}_{+}^n$ and establish its global and quadratic convergence.

The remainder of the paper is organized as follows. In the next section, we
investigate some nice properties of the M-tensor equation (\ref{eqn:M-teq}).
In particular, we propose a criterion to distinguish zero and nonzero elements of a
nonnegative solution of the equation. In Section 3,
we propose a Newton method to get a positive solution to the M-Teq  (\ref{eqn:M-teq}) with $b\in \mathbb {R}_{++}^n$ and
establish its global and quadratic convergence.
In Section 4, we extend the method proposed in Section 3 to the M-Teq (\ref{eqn:M-teq}) with $b\in \mathbb {R}_{+}^n$
and show its  global and quadratic convergence. At last, we do numerical experiments to test the proposed method in Section Section 5.

\section{Properties of M-Tensor Equations}
\label{sec:2}

Throughout this section, we suppose that tensor ${\cal A}\in {\cal T}(m,n)$ is a strong M-tensor.

The following lemma was proved by Li, Guan and Wang \cite{Li-Guan-Wang-18}.
\begin{lemma}\label{lem-g}
If ${\cal A}$ is a strong M-tensor, and the feasible set ${\cal S}$ defined by
\[
{\cal S} \stackrel\triangle {=}\{x\in \mathbb R^n_+|\;F(x)={\cal A}x^{m-1}-b\leq 0\}
\]
is not empty, then ${\cal S}$ has a largest element that is the largest
nonnegative solution to the M-tensor equation $F(x)={\cal A}x^{m-1}-b=0$.
\end{lemma}

As an  application of the last lemma, we have the following proposition.
\begin{proposition}\label{prop:mono-sol}
Let ${\cal A}$ be a strong M-tensor and $b^{(1)},  b^{(2)}\in \mathbb R^n$ satisfy $b^{(2)} \ge  b^{(1)}$.
Suppose that  the M-tensor equation
\begin{equation}\label{temp-g1}
{\cal A}x^{m-1}-b^{(1)}=0
\end{equation}
has a nonnegative solution $x^{(1)}$.  Then the M-tensor equation
\begin{equation}\label{temp-g2}
{\cal A}x^{m-1}-b^{(2)}=0
\end{equation}
has a nonnegative solution $x^{(2)}$ satisfying $x^{(2)}\ge x^{(1)}$.
In particular, if $b^{(1)}>0$, then the unique positive solution $\bar x^{(1)}$ of {\rm(\ref{temp-g1})}
and the unique positive solution $\bar x^{(2)}$ of {\rm(\ref{temp-g2})} satisfies
$\bar x^{(2)}\ge \bar x^{(1)}$.
\end{proposition}
\begin{proof}
Define
\[
{\cal S}_1 \stackrel\triangle {=}\{x\in \mathbb R^n_+|\;{\cal A}x^{m-1}-b ^{(1)} \leq 0\}
\]
and
\[
{\cal S}_2 \stackrel\triangle {=}\{x\in \mathbb R^n_+|\;{\cal A}x^{m-1}-b ^{(2)} \leq 0\}.
\]
Since $b^{(1)}\le b^{(2)}$, we obviously have
\[
{\cal S}_1 \subseteq {\cal S}_2.
\]
By the assumption that (\ref{temp-g1}) has a nonnegative solution,
we claim  from Lemma \ref{lem-g} that the set ${\cal S}_1$ is  nonempty
and has a largest element $\bar x^{(1)}$ that is a solution to the equation (\ref{temp-g1}).
Consequently, the set ${\cal S}_2$ is  nonempty  and has a largest element $x^{(2)}$
that is a solution to the equation (\ref{temp-g2}). It is clear that
\[
x^{(2)}\ge \bar x^{(1)}\ge x^{(1)}.
\]
If $b^{(1)}>0$, then the unique positive solution $\bar x^{(1)}$ is the largest element of
${\cal S}_1$ and  $\bar x^{(2)}$ is the largest element of
${\cal S}_2$. As a result, we have $\bar x^{(2)}\ge \bar x^{(1)}$.
The proof is complete.
\end{proof}

\begin{theorem}\label{th:sol}
Suppose that $\cal A$ is a strong M-tensor. Then the following statements are true.
\begin{itemize}
\item [{\rm(i)}] The tensor equation
\begin{equation}\label{eqn:homo}
{\cal A}x^{m-1}=0
\end{equation}
has a unique solution $x=0$.
\item [{\rm(ii)}] If $-b\in \mathbb R_+^n\backslash \{0\}$, then the M-Teq {\rm(\ref{eqn:M-teq})}
has no nonnegative solutions.

\item [{\rm(iii)}] The following relationship holds
\[
x\circ {\cal A}x^{m-1}=0\qquad \Longleftrightarrow\qquad x=0.
\]
\item [{\rm(iv)}]
It holds that
\[
\lim_{\|x\|\to\infty }\|{\cal A}x^{m-1}\|=+\infty.
\]
\item [{\rm(v)}] For any $b\in \mathbb R^n$, the solution set of the M-tensor equation {\rm(\ref{eqn:M-teq})},
if not empty, is bounded.
\end{itemize}
\end{theorem}
\begin{proof}
Conclusion (i) is trivial because zero is not an eigenvalue of any strong M-tensor.

(ii) Suppose for some $b \le 0$, $b\neq 0$,
the M-Teq (\ref{eqn:M-teq}) has a nonnegative solution $\bar x\neq 0$.
Clearly, $\bar x\ge 0$.
Denote $I=\{i:\bar{x}>0\}$.
Let ${\cal D}$
be a diagonal tensor whose diagonals are
$d_{i\cdots i}=  -b_i \bar x_i^{-(m-1)}$, $\forall i \in I$, and
$d_{i\cdots i}=  0$, $\forall i\notin I$.
Let $\bar {\cal A} ={\cal A}+{\cal D}$. It is obvious that
$\bar {\cal A}$ is a strong M-tensor. Clearly, $\bar {\cal A}_I$ is a strong M-tensor too. However, it holds that
$\bar {\cal A}_I \bar x_I^{m-1}=0$, which yields  a contradiction.

Conclusion (iii) follows from (i) directly  because any principal subtensor of a strong M-tensor
 is a strong M-tensor.

(iv) Suppose on the contrary that there is some sequence $\{x_k\}$ satisfying\\ $\lim_{k\to\infty}\|x_k\|=+\infty$
such that the sequence $\{\|{\cal A}x_k^{m-1}\|\}$ is bounded.
Then we have
\[
\lim _{k\to\infty}\frac {\|{\cal A}x_k^{m-1}\|}{\|x_k\|^{m-1}}=0.
\]
Let $y_k=x_k/\|x_k\|$ and $\bar y$ be an accumulation point of the sequence $\{y_k\}$.
It is easy to see that $\bar y\neq 0$ but ${\cal A}\bar y^{m-1}=0$, which contradicts with (i).

The conclusion (v) is a direct corollary of the conclusion (iv).
\end{proof}

The latter part of this section focuses on the M-Teq (\ref{eqn:M-teq}) with $b \in \mathbb R^n_+$.
We denote
\[
I^+(b)=\{i:\; b_i>0\},\quad\mbox{and}\quad  I^0(b)=\{i:\; b_i=0\}.
\]

We first show the following theorem.

\begin{theorem}\label{th:positive-sol}
Suppose that $\cal A$ is irreducible and is a strong M-tensor. Then
every nonnegative solution of the M-Teq {\rm(\ref{eqn:M-teq})} with $b \in \mathbb R^n_+$ must be positive.
\end{theorem}
\begin{proof}
Suppose on the contrary that the M-Teq (\ref{eqn:M-teq}) with $b \in \mathbb R^n_+$ has a nonnegative
solution $\bar x$ satisfying $I=\{i\;|\; \bar x_i =0\}\neq \emptyset$.
We have for any $i\in I$,
\[
0= \sum_{i_2, \ldots, i_m}a_{i i_2\ldots i_m}\bar x_{i_2}\cdots \bar x_{i_m}-b_i
=\sum_{\{i_2,\ldots, i_m\}\subseteq I_c }a_{i i_2\ldots i_m}\bar x_{i_2}\cdots \bar x_{i_m}-b_i \le 0.
\]
Since $\bar x_j>0$, $\forall j\in I_c$, the last inequality yields $b_i=0$ and
\[
a_{i i_2\ldots i_m}=0,\quad\forall i_2,\ldots, i_m\not\in  I.
\]
It shows that tensor $\cal A$  is reducible with respect to $I$, which yields a contradiction.
\end{proof}

By the proof of the last theorem, we have the following corollary.
\begin{corollary}\label{co-2}
Suppose that $\cal A$ is a strong M-tensor. If the M-Teq {\rm(\ref{eqn:M-teq})} with $b \in \mathbb R^n_+$
has a nonnegative $\bar x$ with zero elements, then $\cal A$ is reducible with respective some $I\subseteq I^0(b)$.
\end{corollary}

The following theorem characterizes a nonnegative solution of the M-Teq (\ref{eqn:M-teq}) with $b \in \mathbb R^n_+$.

\begin{theorem}\label{th:block}
Suppose that $\cal A$ is a strong M-tensor and $b \in \mathbb R^n_+$.
Then the M-Teq  {\rm(\ref{eqn:M-teq})}
has a nonnegative solution with zero elements if and only if $\cal A$ is reducible with respect to some
$I\subseteq I^0(b)$. Moreover, for a nonnegative solution $\bar x$ of the M-Teq  {\rm(\ref{eqn:M-teq})},
$\bar x_i=0$ iff $i\in I$.
\end{theorem}
\begin{proof} The ``only if" part follows from Corollary \ref{co-2} directly.

Suppose  that tensor $\cal A$  is reducible with respect to some $I^1\subseteq I^0(b)$.
It is easy to see that the M-tensor equation (\ref{eqn:M-teq}) has a solution $\bar x$ with $\bar x_{I^1}=0$.
Denote $I^1_c=[n]\backslash I^1$.
Consider the lower dimension M-tensor equation
\begin{equation}\label{temp:th-block}
{\cal A}_{I^1_c}x_{I^1_c}^{m-1}-b_{I^1_c}=0.
\end{equation}
Since $b_{I^1_c}\ge 0$, the last equation has a nonnegative solution $\bar x_{I_c^1}$, which together with $\bar x_{I^1}=0$
forms a nonnegative solution to the M-tensor equation (\ref{eqn:M-teq}).

If ${\cal A}_{I^1_c}$ is irreducible, then $I=I^1$ is the desired index set.
Otherwise,  ${\cal A}_{I^1_c}$ is reducible
with respect to some  $I^2\subset I_c^1$ satisfying $I^2\subseteq b^0_{I_c^1}$. We consider the lower dimensional M-tensor equation (\ref{temp:th-block}).
Following a similar discussion to the above process, we can get another lower dimensional tensor equation
whose nonnegative solution together with some zeros forms a nonnegative solution to (\ref{eqn:M-teq}).
Continuing this process finitely many times, we can get a desirable index set $I\subset I^0(b)$.
\end{proof}

\begin{remark}
The above theorem provides a way to reduce the size of an M-tensor equation with $b\in \mathbb R^n_+$. Specifically,
in the case
tensor $\cal A$  is reducible with respect to some $I\subseteq I^0(b)$,
we can get a solution to (\ref{eqn:M-teq}) by finding  a positive solution to the lower
dimensional tensor equation
\[
{\cal A}_{I_c}x_{I_c}^{m-1}-b_{I_c}=0,
\]
where $I_c=[n]\backslash I$.
\end{remark}

As a direct corollary of the last theorem, we have the following results, which gives
a necessary and sufficient condition for the M-tensor equation (\ref{eqn:M-teq}) with $b\in \mathbb R^n_+$
to have  a positive solution.
\begin{corollary}\label{cor-1}
Suppose that $\cal A$ is a strong M-tensor  and  $b \in \mathbb R^n_+$. Then
there is an index set $I\subseteq I^0(b)$ such that every nonnegative solution to the
following lower dimensional tensor equation with $I_c=[n]\backslash I$
\[
{\cal A}_{I_c}x_{I_c}^{m-1}-b_{I_c}=0
\]
is positive. Moreover, the positive solution $x_{I_c}$ of the last equation together with
$x_{I}=0$ forms a nonnegative solution to the M-Teq {\rm(\ref{eqn:M-teq})}.
\end{corollary}

The following lemma gives another interesting property for an M-Teq.
\begin{lemma}\label{lm:block}
Suppose that ${\cal A}$ is a strong M-tensor and $b \in \mathbb R^n_+$.
Suppose further that every nonnegative solution of the M-Teq {\rm(\ref{eqn:M-teq})} is positive.
 Then there is an index set $J\subseteq I^0(b)$ such that
for each $i\in J$, there are at least one $i_j\not\in J$, $2\le j\le m$ such that
$a_{i i_2\ldots i_m}\neq 0$.
\end{lemma}
\begin{proof}
Let $J_0= I^0(b)$. It is easy to see that
there must be at least one $i\in J_0$ and at least one $i_j\not\in J_0$, $2\le j\le m$ such that
$a_{i i_2\ldots i_m}\neq 0$. Otherwise, the M-Teq has a nonnegative solution $\tilde x$ with $\tilde x_i=0$, $\forall i\in J_0$,
with yields a contradiction.

If $J_0$ does not meet the requirement,
we get an index set $J_1\subset J_0$ consisting of all indices $i\in J_0$ that does not meet the requirement.
If $J_1$ still is not the desired index set, we can further get a small index set
$J_2\subset J_1$. We proceed the process. At last, we get the desired index $J$.
\end{proof}

Based on the last lemma, we can show the nonsingularity property of the Jacobian $F'$ at the
positive solutions.

\begin{theorem}
Suppose that the M-Teq {\rm(\ref{eqn:M-teq})} with a strong M-tensor ${\cal A}$ and $b \in \mathbb R^n_+$ has
a positive solution  $\bar{x}$. Then the Jacobian $F'(\bar x)$ is a nonsingular M-matrix.
Let $f(y)=F(y^{[\frac{1}{m-1}]})$ and $\bar{y}={\bar x}^{[m-1]}$.
Then $f'(\bar y)$ is also a nonsingular M-matrix.
\end{theorem}
\begin{proof}

It is easy to derive for any $y>0$,
$$
f^{\prime}(y)={\cal {A}}(y^{[\frac{1}{m-1}]})^{m-2} \mbox{diag}(y^{[\frac{1}{m-1}-1)}).
$$
It shows that the nonsingularity of $f'(\bar y)$ is the same as
the nonsingularity of $F'(\bar x)$.

Let $J\subseteq I^0(b)$  be the index set specified by Lemma \ref{lm:block} and $I=[n]\backslash J$.
Write the Jacobian matrix $F'(\bar x)$ as
the block form
$$
F'(\bar{x})=(m-1){\cal A}\bar x^{m-2}=\left(\begin{array}{ll}
A_{I I} & A_{I J} \\
A_{J I} & A_{J J}
\end{array}\right).
$$
Since $\bar x$ is a positive solution of (\ref{eqn:M-teq}), it follows from Lemma \ref{lm:block}
that $A_{JI}$ has no zero rows.

That $\bar x$ is a solution to (\ref{eqn:M-teq}) yields $F'(\bar x)\bar x=(m-1)b$.
Writing it as block form, we get
\begin{equation}\label{temp:th}
\left \{\begin{array}{lll}
A_{II}\bar x_I+A_{IJ}\bar x_J= \displaystyle (m-1)b_I, \vspace{1mm} \\
A_{JI}\bar x_I+A_{JJ}\bar x_J= \displaystyle (m-1)b_J.
\end{array}\right.
\end{equation}
It follows from
the last equality of the above system that
$$
A_{J J} \bar{x}_{J}=(m-1)b_{J}-A_{J I} \bar{x}_{I}=-A_{J I} \bar{x}_{I}>0.
$$
Since $A_{JJ}$ is a Z-matrix, the last inequality implies that ${\cal A}_{JJ}$ is a nonsingular M-matrix.
It then suffices to show that the Schur complement $A_{I I}-A_{I J} A_{J J}^{-1} A_{J I}$
is a nonsingular M-matrix.

If $J=I^0(b)$, then $I=I^+(b)$.
We get from the first equality of (\ref{temp:th}),
$$
\left(A_{I I}-A_{I J} A_{J J}^{-1} A_{J I}\right) \bar{x}_{I}=(m-1)b_{I}>0.
$$
Clearly, matrix $A_{I I}-A_{I J} A_{J J}^{-1} A_{J I}$ is a Z-matrix.
Consequently, the last inequality shows that the Schur complement of $A_{JJ}$ is a
nonsingular M-matrix too. Therefore, $A=F'(\bar x)$ is a nonsingular M-matrix.

In the case $J\subset I^0(b)$,we denote $J_1=J$, $I_1=I$ and $A_1=A_{I_1 I_1}-A_{I_1 J_1}
A_{J_1 J_1}^{-1}A_{J_1 I_1}.$
Then to show $F'(\bar x)$ is a nonsingular M-matrix is equivalent to show that the lower
dimensional Z-matrix $A_1$ is a nonsingular M-matrix.
It is clear that $\bar x_{I_1}$ satisfies the lower dimensional system of linear equations
\[
A_1\bar x_{I_1}= (m-1)b_{I_1}.
\]

Similar to above arguments, we can get a partition $(I_2,J_2)$ to the index set $I_1$ that
possesses the same properties as $(I_1,J_1)$.
Repeat the process finitely many times, we can get $J_t=I^0(b_{I_{t-1}})$. As a result, we can
verify that $F'(\bar x)$ is a nonsingular M-matrix.
\end{proof}

\section{A Newton Method for M-Tensor Equation (\ref{eqn:M-teq}) with $b \in \mathbb R^n_{++}$}
\label{sec:3}

In this section, we propose a Newton method to find the unique positive  solution to
(\ref{eqn:M-teq}) with  $b \in \mathbb R^n_{++}$.
Throughout this section, without specification, we always suppose that the following
assumption holds.

\begin{assumption}
Tensor $\cal A$ is a semi-symmetric and  strong M-tensor, and $b \in \mathbb R^n_{++}$.
\end{assumption}

Recently, He, Ling, Qi and Zhou \cite{He-Ling-Qi-Zhou-18}
developed a Newton method for solving the M-Teq (\ref{eqn:M-teq}) with $b\in \mathbb{R}_{++}^n$.
By making a variable transformation $x=y^{[\frac {1}{m}]}$,
they formulated the equation to the following equivalent nonlinear equation:
\[
W(y)=D(y) \cdot F(y^{[\frac {1}{m}]})=D(y)\cdot {\cal A}\Big (y^{[\frac {1}{m}]}\Big ) ^{m-1} -D(y) \cdot b=0,
\]
where $D(y)=\mbox{diag }\Big (y_i^{\frac {1}{m}-1}\Big )$.
The above equation has some nice properties such as
the nonsingularity  of the Jacobian $W'(y)$ for any $y>0$.
In the case where  $\cal A$ is symmetric, the tensor equation (\ref{eqn:M-teq})
is the stationary equation
of the minimization problem
\[
\min \bar f(y)=  \frac 1m {\cal A}\Big (y^{[\frac {1}{m}]}\Big )^m-b^T\Big (y^{[\frac {1}{m}]}\Big )
\]
because the gradient of $\bar f(y)$ is
\[
\nabla \bar f(y) =  \frac 1m W(y)= \frac 1mD(y)\cdot \nabla f(y^{[\frac {1}{m}]}).
\]

In what follows, we propose a Newton method for finding the unique positive
solution of the M-Teq (\ref{eqn:M-teq}). Our idea to develop the Newton method
is similar to but different from that in \cite{He-Ling-Qi-Zhou-18}. Details are given below.

Since our purpose is to get a positive solution of the M-Teq (\ref{eqn:M-teq}),
we restrict $x\in \mathbb R^n_{++}$. Making a variable transformation
 $y=x^{[m-1]}$, we formulate the M-Teq  (\ref{eqn:M-teq}) as
\begin{equation}\label{def:f}
f(y)=F(y^{[\frac {1}{m-1}]}\Big )={\cal A}\Big (y^{[\frac {1}{m-1}]}\Big )^{m-1}-b=0.
\end{equation}
A direct computation gives
\[
f'(y)={\cal A}\Big (y^{[\frac {1}{m-1}]}\Big )^{m-2}\mbox{diag }\Big (y^{[\frac {1}{m-1}-1]}\Big ).
\]
It follows that
\[
f'(y)y={\cal A}\Big (y^{[\frac {1}{m-1}]}\Big )^{m-2}\mbox{diag }\Big (y_i^{[\frac {1}{m-1}-1]}\Big )y
    ={\cal A}\Big (y^{[\frac {1}{m-1}]}\Big )^{m-1}=f(y)+b.
\]

Clearly, the positive solutions of the M-Teq (\ref{eqn:M-teq})
are positive solutions of  the following nonlinear equation:
\begin{equation}\label{equiv-a}
E(y)\stackrel\triangle {=} \mbox{diag }\Big ( y^{[-1]} \Big ) f(y)= \Big ( y^{[-1]} \circ  f(y)\Big ) =0.
\end{equation}
The Jacobian of $E(y)$ is
\begin{eqnarray*}
E'(y) &=& \mbox{diag }\Big ( y^{[-1]} \Big ) f'(y) - \mbox{diag }(f(y)) \mbox{diag }\Big ( y^{[-2]} \Big )\\
    &=& \mbox{diag }\Big ( y^{[-1]} \Big )\Big [ f'(y) - \mbox{diag }(f(y)) \mbox{diag }\Big ( y^{[-1]} \Big )\Big ].
\end{eqnarray*}
It is a non-symmetric Z-matrix. For any $y>0$, it holds that
\begin{eqnarray*}
E'(y)y &=& \mbox{diag }\Big ( y^{[-1]} \Big ) \Big[ f'(y)y - \mbox{diag }(f(y)) \Big ( y^{[-1]} \Big )y\Big ] \\
    &=&  \mbox{diag }\Big ( y^{[-1]} \Big ) \Big[ {\cal A}\Big (y^{[\frac {1}{m-1}]}\Big )^{m-1} - f(y) \Big ]\\
    &=& \mbox{diag }\Big ( y^{[-1]} \Big ) b>0.
\end{eqnarray*}
Consequently, we have got the following proposition.
\begin{proposition}\label{prop:M-matrix}
Let $E:\mathbb R^n_{++}\to \mathbb R$ be defined by {\rm(\ref{equiv-a})}.
For any $y>0$, the Jacobian $E'(y)$ is an M-matrix.
Moreover, the equation {\rm (\ref{equiv-a})} has a unique positive solution that is
the unique positive solution to the M-Teq {\rm (\ref{eqn:M-teq})}.
\end{proposition}

We are going to develop a  Newton method for solving the nonlinear equation (\ref{equiv-a})
in which the Newton direction $d_k$ is the solution to the system of linear equations
\[
E'(y_k) d+E(y_k)=0,
\]
i.e.,
\[
 \mbox{diag }\Big ( y_k^{[-1]} \Big )\Big [ f'(y_k) - \mbox{diag }\Big (\frac {f(y_k)}{y_k} \Big )\Big ]d+
 \mbox{diag }\Big ( y_k^{[-1]} \Big ) f(y_k)=0,
\]
 or equivalently
\begin{equation}\label{sub:Newton}
\Big [ f'(y_k) - \mbox{diag }\Big (\frac {f(y_k)}{y_k} \Big )\Big ]d+ f(y_k)=0
\end{equation}
Here  $\mbox{diag }\Big (\displaystyle \frac {f(y_k)}{y_k}\Big )$ is a diagonal matrix whose
diagonals  are $\displaystyle \frac {f_i(y_k)}{(y_k)_i}$, $i=1,2,\ldots,n$.
We can regard $d_k$ as an inexact Newton method for solving the equation $f(y)=0$ because the
Newton equation (\ref{sub:Newton})
can be written as
\[
f'(y_k)d_k+f(y_k)=r_k,\quad r_k=\mbox{diag }\Big (\frac {f(y_k)}{y_k} \Big )d_k=O(\|f(y_k)\|\,\|d_k\|),
\]
if $y_k>0$ is bounded away from zero.

Let $ y_k(\alpha)=y_k+\alpha d_k$. Then $y_k(\alpha)$ satisfies
\[
\Big [ f'(y_k) - \mbox{diag }\Big (\frac {f(y_k)}{y_k} \Big )\Big ] y_k(\alpha)  =
\Big [ f'(y_k) - \mbox{diag }\Big (\frac {f(y_k)}{y_k} \Big )\Big ]y_k -\alpha f(y_k)=b-\alpha f(y_k).
\]
Since the Jacobian
\[
E'(y_k)= \mbox{diag }\Big ( y_k^{[-1]} \Big )\Big [ f'(y_k) - \mbox{diag }\Big (\frac {f(y_k)}{y_k} \Big )\Big ]
\]
is an M-matrix and $y_k>0$, it is clear that
the matrix
\[
f'(y_k) - \mbox{diag }\Big (\frac {f(y_k)}{y_k} \Big )
\]
is an M-matrix too. Therefore, the inequality $y_k(\alpha)>0$ will be guaranteed if
\begin{equation}\label{step-1}
b-\alpha f(y_k)>0.
\end{equation}
Let
\begin{equation}\label{alpha-max}
\bar \alpha_k ^{\max}= \min\Big \{\frac {b_i}{f_i(y_k)}:\; f_i(y_k)>0\Big \}.
\end{equation}
It is clear that
\[
y_k+\alpha d_k>0,\quad\forall \alpha \in (0,\bar\alpha_k^{\max}).
\]

The iterative process of the Newton method is stated as follows.

\begin{algorithm}\label{algo:Newton} ({\bf Newton's Method})
%\caption{Inexact Newton Method}
\begin{itemize}
\item [] {\bf Initial.} Given constant $\sigma, \rho\in (0,1)$ and $\epsilon > 0$. Select an  initial point $x_0>0$.
such that $y_0=x_0^{[m-1]}$ satisfies $f(y_0)<b$. Let $k=0$.
\item [] {\bf Step 1.} Stop if $\|E(y_k)\|<\epsilon$.
\item [] {\bf Step 2.} Solve the system of linear equations (\ref{sub:Newton})
to get $d_k$.
\item [] {\bf Step 3.}
For given constant $\sigma \in (0,1)$, let $\alpha _k=\max \{\rho^i:\; i=0,1,\ldots\}$ such that
$y_k+\alpha_kd_k>0$ and that the inequality
\begin{equation}\label{search:descent}
\|E(y_k+\alpha _kd_k)\|^2 \le (1-2 \sigma \alpha _k)\|E(y_k)\|^2,\quad\sigma\in (0,1).
\end{equation}
 is satisfied.
\item [] {\bf Step 3.} Let $y_{k+1}=y_k+\alpha _kd_k$. Go to Step 1.
\end{itemize}
\end{algorithm}

\begin{remark}\label{remark-1}
It is easy to see that the inequality (\ref{step-1})  is guaranteed if $f(y_k)<b$.
So, at the beginning, we select $y_0>0$ satisfying
$f(y_0)<b$ and at each iteration, we let $y_{k+1}=y_k+\alpha_kd_k$
such that $f(y_{k+1})<b$. In this way, the inequalities
$f(y_k)<b$ for all $k$.
\end{remark}

\begin{lemma}\label{lm:positive}
Let  $\{y_k\}$ be generated by Algorithm {\rm\ref{algo:Newton}}.
Then there is a  positive constant $c$ such that
 \begin{equation}\label{bound-y}
 y_k  \ge c{\bf e},\quad\forall k\ge 0,
 \end{equation}
 where ${\bf e}=(1,1,\ldots, 1)^T$.
\end{lemma}
\begin{proof}
It is clear that the sequence of the function evaluations $\{\|E(y_k)\|\}$ is decreasing and
hence bounded by some constant $\overline M>0$, i.e.,
\[
\| E(y_k)\|\le \overline M.
\]
Since $\cal A$ is an M-tensor, there is a constant $s>0$ and a nonnegative tensor ${\cal B}\ge 0$ such that
${\cal A}=s {\cal I}-{\cal B}$, where $\cal I$ stands for the identity tensor whose
diagonals are all ones and all other elements are zeros.

By the definition of $E(y)$, we have
\[
E(y)=s {\bf e} - \mbox{diag }(y^{-1}){\cal B}\Big ( y^{[\frac {1}{m-1}]} \Big )^{m-1} - b\circ (y^{[-1]}).
\]
Since ${\cal B}\ge 0$, the last inequality implies for any $y>0$ and each $i\in [n]$
\[
|E_i(y)| \ge  \frac {b_i}{y_i}+  y_i^{-1}\Big ( {\cal B}\Big ( y^{[\frac {1}{m-1}]} \Big )^{m-1}\Big )_i-s
\ge   \frac {b_i}{y_i}-s .
\]
Suppose there is an index $i$ and an infinite set $K$ such that
$\lim_{k\in K,\,k\to\infty}(y_k)_i=0$. We have
\[
\overline M \ge \lim_{k\in K,\,k\to\infty} | E_i(y_k) |\ge  \lim_{k\in K,\,k\to\infty}\frac {b_i}{(y_k)_i}-s
=+\infty,
\]
which yields a contradiction. The contradiction shows that the  inequality in (\ref{lm:positive})
is satisfied with some positive
constant $c$.
\end{proof}

The following theorem establishes the global convergence of the proposed method.
\begin{theorem}\label{th:conv-Newton}
Suppose that the sequence $\{y_k\}$ generated by Algorithm {\rm\ref{algo:Newton}} is bounded.
Then $\{y_k\}$ converges to the unique positive solution to the M-Teq {\rm(\ref{eqn:M-teq})}.
\end{theorem}
\begin{proof}
We first show that the maximum step length $\bar{\alpha}_k^{\max}$ satisfying  (\ref{alpha-max}) can be bounded away from zero.
That is, there is a constant $\bar\alpha$ such that
\begin{equation}\label{bound:alpha}
\bar{\alpha}_k^{\max}\ge \bar\alpha,\quad\forall k\ge 0.
\end{equation}
Indeed, it follows from the last lemma that
\[
\overline M \ge |E_i(y_k)| = \frac {|f_i(y_k)|}{(y_k)_i}.
\]
Since $\{y_k\}$ is bounded, the last inequality implies that for each $i$,
$\{|f_i(y_k)|\}$ is bounded too. By the definition of $\bar\alpha_k^{\max}$, it is bounded away from some
constant $\bar \alpha$. Consequently,
the inequality (\ref{bound:alpha}) is satisfied for all $k\ge 0$.

Next, we show that there is an accumulation point $\bar y$ of $\{y_k\}$ that is a
positive solution to (\ref{eqn:M-teq}).

Suppose $\{y_k\}_K\to\bar y$. By Lemma \ref{lm:positive}, it is clear that $\bar y>0$. Consequently,
$E{\color{red}'}(\bar y)$ is an M-matrix. Moreover,
\[
\lim_{k\in K,\,k\to\infty }d_k=-E'(\bar y)^{-1}E(\bar y)\stackrel\triangle {=}\bar d.
\]
 Without loss of generality, we let
$\lim_{k\in K,\,k\to\infty }\alpha_k=\tilde \alpha$.

If $\tilde\alpha >0$, then the inequality (\ref{search:descent}) shows that
$\{\|E(y_{k+1})\|\}_k\to 0$.

If $\tilde \alpha=0$, then when $k\in K$ is sufficiently large,
the inequality (\ref{search:descent}) is not satisfied with $\alpha_k'=\rho ^{-1}\alpha_k$, i.e.,
\[
\|E(y_k+\alpha _k'd_k)\|^2 - \|E(y_k)\|^2> -2 \sigma \alpha _k'\|E(y_k)\|^2,\quad\sigma\in (0,1).
\]
Dividing both sizes of the inequality by $\alpha_k'$ and then taking limits as $k\to\infty$
with $k\in K$, we get
\[
-\|E(\bar y)\|^2=E(\bar y)^TE'(\bar{y})\bar d\ge -2\sigma\|E(\bar{y})\|^2,
\]
 which implies $E(\bar y)=0$.

Since $\{\|E(y_k)\|\}$ converges, it follows from Lemma \ref{lm:positive} that
every accumulation point of $\{y_k\}$ is a positive solution to (\ref{eqn:M-teq}).
However, the positive solution of (\ref{eqn:M-teq}) is unique.
Consequently, the whole sequence $\{y_k\}$ converges to the unique positive  solution to (\ref{eqn:M-teq}).
\end{proof}

By a standard argument, it is not difficult to show that the convergence rate of $\{y_k\}$ is quadratic.
\begin{theorem}\label{th:conv-Newton2}
Let the conditions in Theorem {\rm\ref{th:conv-Newton}} hold.
Then the convergence rate of $\{y_k\}$ is quadratic.
\end{theorem}

\section{An Extension}
\label{sec:4}
In this section, we extend the Newton method proposed in the last section to the
M-Teq (\ref{eqn:M-teq}) with $b\in \mathbb R^n_+$.
In the case $b$ has zero elements, the M-Teq may have multiple nonnegative or positive
solutions. Our purpose is to find one nonnegative or positive solution of the equation.

We see from the definition of $E(y)$ that the function
$E(y)$ and its Jacobian are not well defined at a point with zero elements.
Therefore, the Newton method proposed in the last section
can not be applied to find a nonnegative solution  with zero elements.
Fortunately, from Corollary \ref{cor-1}, we can  get a nonnegative solution
of (\ref{eqn:M-teq}) by finding a positive solution to a lower dimensional M-Teq.

Without loss of generality, we make the following assumption.

\begin{assumption}\label{ass:B}
Suppose that tensor $\cal A$ is a semi-symmetric and strong M-tensor, and $b\in \mathbb R^n_+$.
Moreover, every nonnegative solution of
the M-Teq {\rm(\ref{eqn:M-teq})} is positive.
\end{assumption}

Similar to the Newton method by He, Ling, Qi and Zhou \cite{He-Ling-Qi-Zhou-18},
we propose another Newton method, which we call a regularized Newton method, such that the method is still
globally and quadratically convergent without assuming the boundedness of the generated sequence
of iterates.

It is easy to see that the M-Teq (\ref{eqn:M-teq}) is equivalent to the following nonlinear equation
\begin{equation}\label{M-teq:equiv}
E(t,y)\stackrel\triangle {=} \left (\begin{array}{c}
t \\y^{[-1]} \circ f(y) +ty
\end{array}\right )=\left (\begin{array}{c}
t \\ \overline E(t,y)
\end{array}\right )
=0,
\end{equation}
where
\[
\overline E(t,y)= E(y) +ty=y^{[-1]} \circ f(y) +ty.
\]
The Jacobian of $E(t,y)$ is
\[
E'(t,y)= \left (\begin{array}{cc}
1 & 0  \\  y &  \overline E'_y(t,y)
\end{array}\right ),
\]
where
\[
\overline E'_y(t,y) = E'(y) +tI
\]
satisfying
\[
\overline{E}'_y(t,y)y= E'(y)y+ty=y^{[-1]}\circ b+ty>0,\quad \forall y\in \mathbb R^n_{++},\;\forall t>0.
\]
Since $\overline E'_y(t,y)$ is a Z-matrix, the last inequality shows that it is a nonsingular M-matrix.
As a result, for any $t>0$ and any $y\in \mathbb R^n_{++}$, the Jacobian
$E'(t,y)$ is nonsingular.

Now, we propose a Newton method for solving the equivalent nonlinear equation (\ref{M-teq:equiv})
to the M-Teq (\ref{eqn:M-teq}). The idea is similar to the Newton method by He, Ling, Qi and
Zhou \cite{He-Ling-Qi-Zhou-18}. Details are given below.

Given constant $\gamma\in (0,1)$. Denote
\[
\theta (t,y)=\frac 12 \|E(t,y)\|^2,\quad \beta (t,y)=\gamma \min\{1,\,\|E(t_k,y_k)\|^2\}.
\]

The subproblem of the method is the following system of linear equations:
\begin{equation}\label {sub:extension}
E'(t_k,y_k)d_k+E(t_k,y_k)=\beta(t_k,y_k){\bf e}_1,
\end{equation}
where ${\bf e}_1=(1,0,\ldots, 0)^T\in \mathbb R^{n+1}$. Let $d_k=(d^t_k, d_k^y)$.

Suppose $t_k\le \bar t$ with $\bar t$ satisfying $\bar t\gamma<1$. Then the Newton direction $d_k$ satisfies
\begin{eqnarray}\label{dir:descent-e}
\nabla \theta (t_k,y_k)^Td_k &=& E(t_k,y_k)^TE'(t_k,y_k)d_k=-\|E(t_k,y_k)\|^2+ \beta(t_k,y_k)\nonumber\\
                             &\le& -(1 -\gamma\bar t)\|E(t_k,y_k)\|^2.
\end{eqnarray}
%\begin{equation}\label{dir:descent-e}
%\nabla \theta (t_k,y_k)^Td_k= E(t_k,y_k)^TE'(t_k,y_k)d_k=-\|E(t_k,y_k)\|^2
%+ t_k \beta(t_k,y_k)\le -(1 -\gamma\bar t)\|E(t_k,y_k)\|^2.
%\end{equation}
As a result, for given constant $\sigma\in (0,1)$, the following inequality
\begin{equation}\label{search:extension}
\theta (t_k+\alpha_kd^t_k, y_k+\alpha_kd^y_k) \le [1- 2\sigma (1 -\gamma\bar t)\alpha_k]\theta (t_k,y_k)
\end{equation}
is satisfied for all $\alpha_k>0$ sufficiently small.

The steps of the method are stated as follows.

\begin{algorithm}

\label{algo:extension}{\bf{Regularized Newton Method}}
\begin{itemize}
\item [] {\bf Initial.} Given constants $\gamma,\sigma, \rho\in (0,1)$, $\epsilon>0$ and $\bar t>0$
such that $\bar t\gamma<1$. Given
 initial point  $x_0>0$ and $t_0=\bar t$. Let $y_0=x_0^{[m-1]}$ and  $k=0$.
\item [] {\bf Step 1.} Stop if $\|E(t_k,y_k)\|\le \epsilon$.
\item [] {\bf Step 2.} Solve the system of linear equations (\ref{sub:extension})
to get $d_k$.
\item [] {\bf Step 3.} Find $\alpha_k=\max\{\rho ^i:\, i=0,1,\ldots \}$
such that $y_k+\rho ^id^y_k>0$ and that
(\ref{search:extension}) is satisfied with $\alpha_k=\rho ^i$.
\item [] {\bf Step 4.} Let $y_{k+1}=y_k+\alpha_kd_k^y$ and $t_{k+1}=t_k+\alpha_kd^t_k$.
\item [] {\bf Step 5.} Let $k:=k+1$. Go to Step 1.
\end{itemize}
\end{algorithm}

Following a similar argument as the proof of Lemma 3.2 of \cite{He-Ling-Qi-Zhou-18},
it is not difficult to get the following proposition. It particularly
shows that the above algorithm is well-defined.
\begin{proposition}\label{prop:well}
Suppose that $\cal A$ is a strong M-tensor and $b\in \mathbb R^n_+$. Then the sequence of
iterates $\{(t_k,y_k)\}$ generated by
Algorithm {\rm \ref{algo:extension}} satisfies
\[
0< t_{k+1}\le t_k\le\bar t
\]
and
\[
t_k> \bar t\beta (t_k,y_k).
\]
In addition, the sequence of function evaluations $\{\theta (t_k,y_k)\}$ is decreasing.
\end{proposition}

Since $\cal A$ is an M-tenor, there are a constant $s>0$ and a nonnegative tensor ${\cal B}=(b_{i_1\ldots i_m})$ such that
${\cal A}=s{\cal I}-{\cal B}$, where $\cal I$ is the identity tensor whose diagonal entities are all ones and all
other elements are zeros.
By the definition of $E(y)$, it is easy to get
\[
E(y)=s {\bf e} -y^{[-1]}\circ {\cal B}\Big (y^{[\frac {1}{m-1}]}\Big )^{m-1}-y^{[-1]}\circ b.
\]

\begin{lemma}\label{lm:l-bound}
Suppose that $\cal A$ is a strong M-tensor and $b\in \mathbb R^n_+$. Then
the sequence of iterates $\{y_k\}$ generated by
Algorithm {\rm\ref{algo:extension}} is bounded away from zero. In other words, there is a constant
$\eta>0$ such that
\[
(y_k)_i\ge\eta,\quad\forall k\ge 0,\;\forall i=1,2,\ldots,n.
\]
\end{lemma}
\begin{proof}
Suppose that there is an index $i$ and a subsequence $\{y_k\}_K$ such that
$\lim _{k\to\infty,\,k\in K}(y_k)_i=0$. Without loss of generality, we suppose
$\{y_k\}_K\to \bar y$, where some elements of $\bar y$ may be $+\infty$.
Denote $I=\{i: \bar y_i=0\}$ and $I_c=[n]\backslash I_c$. Since Let $\{\theta (t_k,y_k)\}$
is decreasing, it is bounded and so is the sequence $\{\|E(t_k,y_k)\|\}$.
Let $C>0$ be an upper bound of
the sequence $\{\|\overline{E}(t_k,y_k)\|\}$.

For each $i\in I$, it holds that
\begin{eqnarray*}
C &\ge & |\overline{E_i}(t_k,y_k)| = \Big | \frac 1 {(y_k)_i}
    \sum _{i_2,\ldots,i_m}a_{i i_2\ldots i_m} \Big ( (y_k)_{i_2}^{\frac {1}{m-1}}\cdots (y_k)_{i_m}^{\frac {1}{m-1}}\Big )
        -  \frac {b_i} {(y_k)_i} +t_k (y_k)_i\Big |\\
   & =&  \Big | s - \frac 1 {(y_k)_i}
    \sum _{i_2,\ldots,i_m}b_{i i_2\ldots i_m} \Big ( (y_k)_{i_2}^{\frac {1}{m-1}}\cdots (y_k)_{i_m}^{\frac {1}{m-1}}\Big )
        -  \frac {b_i} {(y_k)_i} +t_k (y_k)_i\Big |\\
    &\ge &
    \sum _{i_2,\ldots,i_m}b_{ii_2\ldots i_m} \left ( \frac {(y_k)_{i_2}} {(y_k)_i}
            \cdots \frac {(y_k)_{i_m}} {(y_k)_i}\right )^{\frac {1}{m-1}}
        +  \frac {b_i} {(y_k)_i} -t_k (y_k)_i -s\\
    &\ge & \sum _{i_2,\ldots,i_m \in I_c}b_{ii_2\ldots i_m} \left ( \frac {(y_k)_{i_2}} {(y_k)_i}
            \cdots \frac {(y_k)_{i_m}} {(y_k)_i}\right )^{\frac {1}{m-1}}
        +  \frac {b_i} {(y_k)_i} -t_k (y_k)_i -s .
\end{eqnarray*}
Notice that for any $i\in I_c$, $\bar y_i>0$.
Since $t_k\le \bar t$ and $(y_k)_i\to 0$, as $k\to\infty$ with $k\in K$, the last inequality implies
$b_i=0$ and $a_{ii_2\ldots i_m}=b_{ii_2\ldots i_m}=0$, $\forall i_2,\ldots,i_m \in I_c$.
It means that tensor ${\cal A}$ is reducible with respect to index set $I$.
It then follows from Theorem \ref{th:block} that the M-Teq (\ref{eqn:M-teq}) has a nonnegative solution
that has zero elements. It is a contradiction. The contradiction shows
that $\{y_k\}$ is bounded away from zero.
\end{proof}

\begin{lemma}\label{lm:u-bound}
Suppose that $\cal A$ is a strong M-tensor and $b\in \mathbb R^n_+$.
If there is a $\tilde t>0$ such that $t\ge \tilde t$,
then the sequence of iterates $\{y_k\}$ generated by Algorithm {\rm\ref{algo:extension}} is bounded.
\end{lemma}
\begin{proof}
Denote by $i_k$ the index satisfying $(y_k)_{i_k}=\|y_k\|_{\infty}$. Since
$\{\theta (t_k,y_k)\}$ has an upper bound, so is $\{\|\overline{E}(t_k,y_k)\|\}$.
Let $C$ be an upper bound of $\{\|\overline{E}(t_k,y_k)\|\}$.
% and by $C$ the upper bound of $\{\theta (t_k,y_k)\}$.
It is clear that
\[
\Big |\sum _{i_2,\ldots,i_m}a_{i_ki_2\ldots i_m} \left ( \frac {(y_k)_{i_2}} {(y_k)_{i_k}}
            \cdots \frac {(y_k)_{i_m}} {(y_k)_{i_k}}\right )^{\frac {1}{m-1}}\Big |
            \le \sum _{i_2,\ldots,i_m} |a_{i_ki_2\ldots i_m}|\stackrel\triangle {=} \tilde a_{i_k}
\]
is bounded.
Therefore, we obtain
\begin{eqnarray*}
C &\ge & \|\overline{E}(t_k,y_k)\|\\
    &\ge & \Big | \frac 1 {(y_k)_{i_k}}
    \sum _{i_2,\ldots,i_m}a_{i_ki_2\ldots i_m} \Big ( (y_k)_{i_2}^{\frac {1}{m-1}}\cdots (y_k)_{i_m}^{\frac {1}{m-1}}\Big )
        -  \frac {b_{i_k}} {(y_k)_{i_k}} +t_k (y_k)_{i_k} \Big |\\
    &\ge &  t_k (y_k)_{i_k}- \tilde{a}_{i_k} -  \frac {b_{i_k}} {(y_k)_{i_k}}.
\end{eqnarray*}
The last inequality together with $t_k\ge \tilde t$ implies that $\{\|y_k\|\}$ is bounded.
\end{proof}

The following theorem establishes the global convergence of Algorithm \ref{algo:extension}.
\begin{theorem}\label{th:conv-extension}
Suppose that $\cal A$ is a strong M-tensor and $b\in \mathbb{R}_{+}^n$. Then every accumulation point of
the sequence of iterates $\{(t_k,y_k)\}$ generated by
Algorithm {\rm\ref{algo:extension}} is a positive solution to the M-Teq {\rm(\ref{eqn:M-teq})}.
\end{theorem}
\begin{proof}
It suffices to show that the sequence $\{\theta (t_k,y_k)\}$ converges to zero by contradiction.
Suppose on the contrary that there is a constant $\delta>0$ such that
$\theta (t_k,y_k)\ge\delta$, $\forall k\ge 0$. Then
\[
\tilde t \stackrel\triangle {=} \lim_{t\to\infty} t_k\ge
\bar t \lim_{t\to\infty} \beta (t_k,y_k) \ge \bar t \gamma \min \{1,2\delta\}
>0.
\]
By Lemma \ref{lm:u-bound}, $\{y_k\}$ is bounded. Let the subsequence $\{y_k\}_K$ converges to
some point $\bar y$. Lemma \ref{lm:l-bound} ensures $\bar y>0$. It is easy to show that
the Jacobian $E'(\tilde t,\bar y)$ is a nonsingular M-matrix.
Consequently, $\{d_k\}_K$ is bounded. Without loss of generality, we suppose  $\{d_k\}_K$
converges to some $\bar d$. Since $\bar y>0$, there is a constant $\alpha^{\min}>0$
such that $y_k+\alpha_kd_k>0$, $\forall \alpha_k\in (0,\alpha^{\min})$.
Let $\bar\alpha=\mathop {\lim\inf}_{k\to\infty,\,k\in K}\alpha_k$. If $\bar\alpha>0$, the line search
condition (\ref{search:extension}) implies $\theta(\bar y, \tilde t)=0$. If $\bar\alpha =0$,
then when $k$ is sufficiently large, the inequality (\ref{search:extension}) is not satisfied
with $\alpha_k'=\alpha_k\rho^{-1}$, i.e.,
\[
\theta (t_k+\alpha_k'd^t_k, y_k+\alpha_k'd^y_k) - \theta (t_k,y_k) \ge - 2\sigma (1 -\gamma\bar t)\alpha_k'\theta (t_k,y_k).
\]
Dividing both sizes of the last inequality by $\alpha_k'$ and then taking limits
as $k\to\infty$ with $k\in K$, we get
\[
\nabla \theta (\tilde t, \bar y)^T\bar d \ge - 2\sigma (1 -\gamma\bar t)\theta (\tilde t, \bar y).
\]
On the other hand, by taking limits in both sizes of (\ref{dir:descent-e}) as $k\to\infty$ with $k\in K$,
we obtain
\[
\nabla \theta (\tilde t, \bar y)^T\bar d \le - 2(1 -\gamma\bar t)\theta (\tilde t, \bar y).
\]
Since $\sigma\in (0,1)$, the last two inequalities implies $\theta (\bar y,\tilde t)=0$, which yields a
contradiction. As a result, we claim that $\{\theta (t_k,y_k)\}$ converges to zero.
The proof is complete.
\end{proof}

 The last theorem has shown that every accumulation is a positive solution to the M-Teq (\ref{eqn:M-teq}).
However, it does not the existence of the accumulation point.
The following theorem shows that the sequence $\{y_k\}$
is bounded. As a result, it ensure the existence of the accumulation point.
\begin{theorem}\label{th:bound-yk}
Suppose that $\cal A$ is a strong M-tensor and $b\in\mathbb{R}_{+}^n$. Then
the sequence  $\{y_k\}$ generated by
Algorithm {\rm\ref{algo:extension}} is bounded.
\end{theorem}
\begin{proof}
First, similar to the proof of Lemma \ref{lm:u-bound},
it is not difficult to  show that the sequence $\{t_ky_k\}$ is bounded.

Case (i), $\{t_ky_k\}\to 0$. Since $\{\theta (y_k,t_k)\}\to 0$, we immediately have
$\{E(y_k)\}\to 0$. Denote $\mu_k=\|y_k\|_{\infty}$, $\tilde y_k=\mu_k^{-1}y_k$
and $\tilde b_k=\mu_k^{-1}b$.
Clearly, the sequence $\{\tilde y_k\}$ is bounded.
If $\{y_k\}$ is unbounded, then there  is a subsequence $\{\mu_k\}_K\to\infty$,
and hence $\{\tilde b_k\}_K\to 0$.  Without loss of generality, we suppose that
the subsequence
$\{\tilde y_k\}_K$ converges to some $\tilde y\ge 0$.
Denote by $J$ the set of indices $i$ satisfying $\tilde y_i>0$. Obviously, $J\neq\emptyset$.

For some $i\in J$, satisfies $y_i=\|y_k\|_\infty$, we have
\begin{eqnarray*}
|\overline{E}_i(y_k,t_k)| &=& \Big | \frac 1 {(y_k)_i}
    \sum _{i_2,\ldots,i_m}a_{i i_2\ldots i_m} \Big ( (y_k)_{i_2}^{\frac {1}{m-1}}\cdots (y_k)_{i_m}^{\frac {1}{m-1}}\Big )
        -  \frac {b_i} {(y_k)_i} +t_k (y_k)_i\Big |\\
   & =&  \Big |
    \sum _{i_2,\ldots,i_m}a_{i i_2\ldots i_m} \Big ( (\tilde y_k)_{i_2}^{\frac {1}{m-1}}\cdots
    (\tilde y_k)_{i_m}^{\frac {1}{m-1}}\Big )
        -  (\tilde b_k)_i +t_k (y_k)_i\Big | .
\end{eqnarray*}
Taking limits in both sizes of the equality as $k\to\infty$ with $k\in K$ yields
\[
0=\sum _{i_2,\ldots,i_m}a_{i i_2\ldots i_m} \Big (\tilde y_{i_2}^{\frac {1}{m-1}}\cdots
\tilde y_{i_m}^{\frac {1}{m-1}}\Big )
=\sum _{i_2,\ldots,i_m\in J}a_{i i_2\ldots i_m} \Big (\tilde y_{i_2}^{\frac {1}{m-1}}\cdots
\tilde y_{i_m}^{\frac {1}{m-1}}\Big ),
\]
Let ${\cal A}_J$ be the principal subtensor of $\cal A$ with elements
$a_{i_1 i_2\ldots i_m}$, $\forall i_1, i_2,\ldots, i_m\in J$. It is a strong M-tensor but
${\cal A}_J\Big (\tilde y^{[\frac {1}{m-1}]}\Big )_J^{m-1}=0$ with $\tilde y\neq 0$.
It is a contradiction. Consequently, $\{y_k\}$ is bounded.

Case (ii), there are at least one $i$ such that $\mathop{\lim\inf}_{k\to\infty} t_k(y_k)_i>0$.
In other words, there is a subsequence $\{t_ky_k\}_K\to \tilde y\ge 0$ such that
$\tilde y_i>0$ for at least one $i$. Again, denote by $J$ the set of indices for satisfying $\tilde y_i>0$.
Since $\{t_k\}\to 0$, it is easy to see that
\[
\lim_{k\to\infty,\,k\in K} (y_k)_i=+\infty,\quad \forall i\in J.
\]
Denote $\tilde y_k=t_ky_k$. Similar to Case (i), we can get
We derive for any $i\in J$
\begin{eqnarray*}
|\overline{E}_i(y_k,t_k)| &=& \Big | \frac 1 {(y_k)_i}
    \sum _{i_2,\ldots,i_m}a_{i i_2\ldots i_m} \Big ( (y_k)_{i_2}^{\frac {1}{m-1}}\cdots (y_k)_{i_m}^{\frac {1}{m-1}}\Big )
        -  \frac {b_i} {(y_k)_i} +t_k (y_k)_i\Big |\\
   & =&  \Big | \frac 1 {(\tilde y_k)_i}
    \sum _{i_2,\ldots,i_m}a_{i i_2\ldots i_m} \Big ( (\tilde y_k)_{i_2}^{\frac {1}{m-1}}\cdots
    (\tilde y_k)_{i_m}^{\frac {1}{m-1}}\Big )
        -   \frac {b_i} {(y_k)_i} + (\tilde y_k)_i\Big | .
\end{eqnarray*}
Taking limits in both sizes of the equality as $k\to\infty$ with $k\in K$ yields
\begin{small}
\[
0=\sum _{i_2,\ldots,i_m}a_{i i_2\ldots i_m} \Big (\tilde y_{i_2}^{\frac {1}{m-1}}\cdots \tilde y_{i_m}^{\frac {1}{m-1}}\Big )
+\tilde y_i
=\sum _{i_2,\ldots,i_m\in J}a_{i i_2\ldots i_m} \Big (\tilde y_{i_2}^{\frac {1}{m-1}}\cdots \tilde y_{i_m}^{\frac {1}{m-1}}\Big )
+\tilde y_i,
\forall i\in J.
\]
\end{small}
It contradicts Theorem \ref {th:sol} (ii).

The proof is complete.
\end{proof}
Similar to theorem 3.3 of \cite{He-Ling-Qi-Zhou-18}, we have the following theorem.
\begin{theorem}
Let the conditions in Assumption {\rm\ref{ass:B}} hold, then the sequence of iterates $\{t_k,y_k\}$
generated by Algorithm {\rm\ref{algo:extension}} converges to a positive solution of the equation {\rm\ref{M-teq:equiv}}.
 And the convergence rate is quadratic.
\end{theorem}

\section{Numerical Results}
\label{sec:5}

In this section, we do numerical experiments to test the effectiveness of the proposed methods.
We implemented our methods in Matlab R2015b and ran the codes on a personal computer with 2.30 GHz CPU and 8.0 GB RAM.
We used a tensor toolbox \cite{Bader-Kolda-15} to proceed tensor computation.

While do numerical experiments, similar to \cite{Han-17,He-Ling-Qi-Zhou-18}, we solved the tensor equation
\[
\hat{F}(x)=\hat{\cal {A}}x^{m-1}-\hat{b}=0
\]
instead of the tensor equation (\ref{eqn:M-teq}),
where $\hat{\cal{A}}:=\cal{A}/\omega$ and $\hat{b}:=b/\omega$ with $\omega$ is the largest value among
the absolute values of components of $\cal{A}$ and $b$. The stopping criterion  is set to
\[
\|\hat{F}(x_k)\|\leq 10^{-10}.
\]
or the number of iteration reaches to 300. The latter case means that the method is failure for the
 problem.

{\bf Problem 1. }\cite{Ding-Wei-16}  We solve tensor equation (\ref{eqn:M-teq}) where ${\cal A}$ is a symmetric strong
M-tensor of order $m$ $(m=3,4,5)$ in the form ${\cal A}=s{\cal I}-{\cal B}$, where tensor ${\cal B}$ is symmetric
whose entries are uniformly distributed in $(0,1)$, and
\[
s=(1+0.01)\cdot\max_{i=1,2,\ldots,n}({\cal B}{\bf e}^{m-1})_i,
\]
where ${\bf e}=(1,1,\ldots,1)^T$.

{\bf Problem 2. }\cite{Xie-Jin-Wei-2017}  We solve tensor equation (\ref{eqn:M-teq}) where $\cal A$ is a symmetric strong M-tensor
of order $m$ $(m=3,4,5)$ in the form ${\cal A}=s{\cal I}-{\cal B}$, and tensor ${\cal B}$ is a nonnegative tensor with
\[
b_{i_1i_2\ldots i_m}=|\mathrm{sin}(i_1+i_2+\ldots +i_m)|,
\]
and $s=n^{m-1}$.

{\bf Problem 3. }\cite{Ding-Wei-16} Consider the ordinary differential equation
$$
 \frac{d^2x(t)}{dt^2}=-\frac{GM}{x(t)^2},\quad t\in(0, 1),
 $$
with Dirichlet's boundary conditions
$$
x(0)=c_0, \quad x(1)=c_1,
$$
where $G\approx 6. 67\times 10^{-11}Nm^2/kg^2$ and $M\approx 5. 98\times 10^{24}$ is the
gravitational constant and the mass of the earth.

Discretize the above equation, we have
\[\left\{
\begin{array}{l}
x^3_1=c^3_0, \\
2x^3_i-x^2_ix_{i-1}-x^2_ix_{i+1}=\frac{GM}{(n-1)^2}, \quad i=2, 3, \cdots, n-1, \\
x^3_n=c^3_1.
\end{array}
\right.
\]
It is a tensor equation, i.e.,
$${\cal A}x^3=b, $$
where ${\cal A}$ is a 4-th order M tensor whose entries are
\[
\left\{
\begin{array}{l}
a_{1111}=a_{nnnn}=1, \\
a_{iiii}=2, \quad i=2, 3, \cdots,  n-1, \\
a_{i(i-1)ii}=a_{ii(i-1)i}=a_{iii(i-1)}=-1/3, \quad i=2, 3, \cdots,  n-1, \\
a_{i(i+1)ii}=a_{ii(i+1)i}=a_{iii(i+1)}=-1/3, \quad i=2, 3, \cdots,  n-1, \\
\end{array}
\right.
\]
and b is a positive vector with
\[
\left\{\begin{array}{l}
b_1=c^3_0, \\
b_i=\frac{GM}{(n-1)^2}, \quad i=2, 3, \cdots,  n-1, \\
b_n=c^3_1.
\end{array}\right.
\]

{\bf Problem 4. }\cite{Li-Guan-Wang-18}  We solve tensor equation (\ref{eqn:M-teq}) where $\cal A$ is a non-symmetric
strong M-tensor of order $m$ $(m=3,4,5)$ in the form ${\cal A}=s{\cal I}-{\cal B}$, and tensor ${\cal B}$ is nonnegative
tensor whose entries are uniformly distributed in $(0,1)$.
The parameter $s$ is set to
\[
s=(1+0.01)\cdot\max_{i=1,2,\ldots,n}({\cal B}{\bf e}^{m-1})_i.
\]

{\bf Problem 5.} We solve tensor equation (\ref{eqn:M-teq}) where $\cal A$ is a lower triangle strong
M-tensor of order $m$ $(m=3,4,5)$ in the form ${\cal A}=s{\cal I}-{\cal B}$, and tensor ${\cal B}$ is
a strictly lower triangular nonnegative
tensor whose entries are uniformly distributed in $(0,1)$.
The parameter $s$ is set to
\[
s=(1-0.5)\cdot\max_{i=1,2,\ldots,n}({\cal B}{\bf e}^{m-1})_i.
\]

For Problem 4 and 5, we need to semi-symmetrize the tensor ${\cal A}$, i.e., find a
semi-symmetric tensor $\tilde{{\cal A}}$ such that
\[
{\cal A}x^{m-1}=\tilde{{\cal A}}x^{m-1}.
\]
The time of semi-symmetrize the tensor is not included in CPU time.

We first test the performance of the Inexact Newton method. We set the start point $x_0=\varepsilon\textbf{e}$, where parameter $\varepsilon$ is selected to satisfy $f(y_0) < b$.
We set the parameter $\sigma=0.1$ and $\rho=0.5$. And $b$ is uniformly distributed in $(0, 1)$ except the $b$ in the problem 3.

For the stability of numerical results, we test the problems of different sizes.
For each pair $(m, n)$, we randomly generate 100 tensors $\cal A$ and $b$.
In order to test the effectiveness of the proposed method, we compare Inexact Newton method
with the QCA method in \cite{He-Ling-Qi-Zhou-18}.
We take parameters $\delta=0.5, \gamma=0.8, \sigma=0.2, \bar{t}=2/(5\gamma)$ as the same as in \cite{He-Ling-Qi-Zhou-18}.
The results are listed in Tables \ref{c1}, where
\[
\mbox{IR}=\frac{\mbox{the number of iteration steps of the Inexact Newton method}}
{\mbox{the number of iteration steps of the QCA method}}
\]
and
\[
\mbox{TR}=\frac{\mbox{the CPU time used by the Inexact Newton method}}{\mbox{the CPU time used by the QCA method}}.
\]

{
\begin{table}[H]
%\scriptsize
\caption{\small  Comparison between Inexact Newton method and QCA method with $b\in \mathbb{R}_{+}^n$.}\label{c1}
\resizebox{\textwidth}{!}{
\centering
\begin{tabular}{c|c|cccc|ccc|cc}\hline\hline
 $$ & $(m,n)$ &(3,10) & (3,100) & (3,300) & (3,500) &(4,10) & (4,50) & (4,100) & (5,10) & (5,30)\\ \hline
IR & Problem 1 &89.2\%	&	91.5\%	&	91.5\%	&	91.0\%	&	93.0\%	&	93.7\%	&	94.3\%	&	95.2\%	&	96.3\%\\
& Problem 2 &91.0\%	&	91.4\%	&	90.2\%	&	90.5\%	&	95.7\%	&	94.8\%	&	93.1\%	&	97.2\%	&	96.2\%\\
 & Problem 3 &-		&	-	    &	-	    &	-       &   11.1\%	&	9.1\%	&	8.3\%	&	-	    &	-\\
 & Problem 4 &91.8\%	&	91.2\%	&	91.3\%	&	90.5\%	&	94.4\%	&	94.7\%	&	93.2\%	&	97.1\%	&	93.9\%\\
 & Problem 5 &89.8\%	&	90.4\%	&	89.6\%	&	89.3\%	&	95.2\%	&	91.5\%	&	93.0\%	&	95.1\%	&95.6\%\\\hline

TR & Problem 1 &48.0\%	&	66.6\%	&	87.7\%	&	88.2\%	&	67.3\%	&	92.3\%	&	94.0\%	&	80.0\%	&	97.0\%\\
& Problem 2 &50.0\%	&	73.8\%	&	88.8\%	&	88.8\%	&	67.4\%	&	94.0\%	&	93.7\%	&	79.0\%	&	96.6\%\\
	& Problem 3 &-        &	-    	&		-   &	-   	&	20.3\%	&	15.4\%	&	14.2\%	&	-	    &	-\\
& Problem 4 &54.1\%	&	73.3\%	&	89.6\%	&	89.6\%	&	66.0\%	&	94.7\%	&	94.1\%	&	74.6\%	&	95.4\%\\
& Problem 5 &45.7\%	&	74.1\%	&	87.4\%	&	88.5\%	&	59.1\%	&	92.3\%	&	95.0\%	&	74.6\%	&	99.4\%\\ \hline\hline	
\end{tabular}}

\end{table}
}

We then test the effectiveness of the Regularized Newton method. We set the initial point
$x_0=0.1*\bf{e}$ and $b\in \mathbb{R}_{+}^n$ has $0$ zero elements except the problem 3.
We first generate a vector $b^0\in\mathbb{R}^n$ whose elements are uniformly distributed in $(0,1)$, then we set
\[
 b_i=\left\{\begin{array}{ll}
b_i^0, \quad & \mbox{if } b_i^0\leq 0.6,\\
0,  \quad & \mbox{if } b_i^0> 0.6.
\end{array}\right.
\]
to get a vector $b\in \mathbb{R}_{+}^n$.
In order to get the positive solution of the problem 5, the first component of vector $b$ can't be equal to 0, so we set the first component $b_1=0.1.$

We compare the Regularized Newton Method with QCA method. We take the parameters $\sigma=0.1, \rho=0.8, \gamma=0.9$ and $\bar{t}=0.01$ in Regularized Newton Method and the parameters in QCA method is the same as above.
The results are listed in Tables \ref{c2}, where
\[
\mbox{IR}=\frac{\mbox{the number of iteration steps of the Regularized Newton method}}
{\mbox{the number of iteration steps of the QCA method}}
\]
and
\[
\mbox{TR}=\frac{\mbox{the CPU time used by the Regularized Newton method}}{\mbox{the CPU time used by the QCA method}}.
\]
{
\begin{table}[H]
\caption{\small  Comparison between Regularized Newton method and QCA method with $b\in \mathbb{R}_{++}^n$.}\label{c2}
\centering
%\footnotesize
\resizebox{\textwidth}{!}{
\begin{tabular}{c|c|cccc|ccc|cc}\hline\hline
 $$ & $(m,n)$ &(3,10) & (3,100) & (3,300) & (3,500) &(4,10) & (4,50) & (4,100) & (5,10) & (5,30)\\ \hline
IR & Problem 1  &	92.4\%	&	59.7\%	&	71.6\%	&	67.4\%	&	93.2\%	&	67.2\%	&	59.5\%	&	95.7\%	&78.4\%\\
   & Problem 2  &	83.9\%	&	61.5\%	&	50.3\%	&	49.4\%	&	89.1\%	&	56.0\%	&	59.3\%	&	87.0\%	&59.4\%\\
   & Problem 3  &		-   &		-   &		-   &	-   	&	83.3\%	&	80.0\%	&	81.0\%	&	-	    &	\\
   & Problem 4  &	94.3\%	&	65.8\%	&	58.1\%	&	60.9\%	&	95.1\%	&	67.7\%	&	53.8\%	&	95.5\%	&	76.5\%\\
   & Problem 5  &	80.0\%	&	81.0\%	&	81.1\%	&	81.4\%	&	75.4\%	&	77.6\%	&	76.4\%	&	72.4\%	&74.0\%\\ \hline
TR & Problem 1  &	80.0\%	&	78.7\%	&	89.2\%	&	82.4\%	&	93.2\%	&	77.2\%	&	71.6\%	&	86.8\%	&97.8\%\\
   & Problem 2  &	72.7\%	&	72.9\%	&	61.3\%	&	60.4\%	&	87.5\%	&	61.4\%	&	64.0\%	&	94.1\%	&65.3\%\\
   & Problem 3  &	-	    &	-    	&		-   &		-   &	76.5\%	&	97.6\%	&	98.1\%	&		-   &-\\
   & Problem 4  &	80.6\%	&	81.3\%	&	65.7\%	&	76.7\%	&	92.3\%	&	78.1\%	&	65.0\%	&	97.4\%	&99.6\%\\
   & Problem 5  &	72.1\%	&	94.9\%	&	89.6\%	&	91.0\%	&	86.9\%	&	79.3\%	&	77.7\%	&	82.5\%	&	80.8\%\\\hline\hline	
\end{tabular}}

\end{table}
}

The datas in Table \ref{c1} and \ref{c2} show that for all test problems the
Inexact Newton method and the Regularized Newton method are better than QCA method in terms of
the number of iterations and CPU time.
It is worth noting that although the QCA method in \cite{He-Ling-Qi-Zhou-18} does not established the convergence property in the case of $b\in \mathbb{R}^n_{+}$,
we find that in the case of $b\in \mathbb{R}^n_{+}$, the QCA method can still find the solution of the problem successfully.
For the convenience of readers, we only list the relative results. More detailed numerical results
can be found in the Appendix.

\appendix
\section{Detailed Numerical Results}
In this section, we list the detailed numerical results of the proposed methods compared with QCA method.
The results are listed in Tables \ref{IN1}, \ref{IN2}, \ref{IN3}, \ref{IN4}, \ref{IN5}, \ref{EN1},
\ref{EN2}, \ref{EN3}, \ref{EN4} and \ref{EN5}, where the columns `Iter', `Time', 'Res' and 'Ls-iter' stand for the total number of iterations, the computational time (in second) used for the method, the residual
$\|\hat{\cal{A}}x_k^{(m-1)}-\hat{b}\|$ and the total number of iterations of linear search.
{
\begin{table}[H]
\caption{\small  Comparison between Inexact Newton method and QCA method on Problem 1 .}\label{IN1}
\centering
\begin{tabular}{c|cccc|cccc}\hline\hline
 $$ & \multicolumn{4}{c|}{Inexact Newton method} & \multicolumn{4}{c}{QCA} \\ \hline
$(m,n)$   & Iter& Time    & Res   & Ls-iter      & Iter & Time & Res  & Ls-iter \\ \hline
(3,10)    &6.6  &0.00024  &8.5E-12  &0  &7.4   &0.00050  &6.1E-12 & 0\\	
(3,100)   &9.7  &0.00829  &9.0E-12  &0  &10.6  &0.01244  &1.0E-11 & 0 \\ 				
(3,300)   &11.9 &0.32238  &9.8E-12  &0  &13.0  &0.36766  &1.7E-11 & 0 \\ 				
(3,500)   &12.1 &1.44961  &5.1E-12  &0  &13.3  &1.64275  &7.4E-12 & 0  \\ 				
\hline 				
(4,10)    &6.6  &0.00033  &5.0E-12  &0  &7.1  &0.00049  &9.2E-12  &0\\ 			
(4,50)    &8.9  &0.04552  &1.2E-11  &0  &9.5  &0.04931  &1.4E-11  &0\\			
(4,100)   &10.0 &0.77240  &1.3E-11  &0  &10.6 &0.82138  &6.7E-12  &0\\ 				
\hline 				
(5,10)    &6.0  &0.00052  &1.3E-11  &0 &6.3  &0.00065  &1.3E-11  &0\\ 				
(5,30)    &7.9  &0.15599  &8.9E-12  &0 &8.2  &0.16078  &1.2E-11  &0\\ \hline\hline				

\end{tabular}

\end{table}
}

{
\begin{table}
\caption{\small   Comparison between Inexact Newton method and QCA method on Problem 2.}\label{IN2}
\centering
\begin{tabular}{c|cccc|cccc}\hline\hline
 $$ & \multicolumn{4}{c|}{Inexact Newton method} & \multicolumn{4}{c}{QCA} \\ \hline
$(m,n)$   & Iter& Time    & Res   & Ls-iter      & Iter & Time & Res  & Ls-iter \\ \hline
(3,10)    &7.1  &0.00020  &1.0E-11& 0   &7.8   &0.00040  &4.7E-12  & 0\\	
(3,100)   &9.6  &0.00899  &5.7E-12& 0   &10.5  &0.01218  &9.1E-12 & 0 \\	
(3,300)   &11.9 &0.32718  &8.6E-12& 0   &13.2  &0.36855  &1.1E-11 & 0 \\	
(3,500)   &12.4 &1.49714  &8.0E-12& 0   &13.7  &1.68625  &8.3E-12& 0  \\	
\hline 				
(4,10)    &6.7  &0.00029  &8.6E-12  &0  &7.0  &0.00043  &9.8E-12  &0\\		
(4,50)    &9.1  &0.04654  &7.8E-12  &0  &9.6  &0.04950  &1.4E-11  &0\\		
(4,100)   &9.5  &0.74050  &1.5E-11  &0  &10.2 &0.79047  &1.7E-11  &0\\ 		
\hline 				
(5,10)    &6.9  &0.00049  &6.5E-12  &0 &7.1  &0.00062  &1.3E-11  &0\\
(5,30)    &7.6  &0.15073  &1.0E-11  &0 &7.9  &0.15600  &1.4E-11  &0\\ \hline\hline				

\end{tabular}

\end{table}
}

{
\begin{table}
\caption{\small   Comparison between Inexact Newton method and QCA method on Problem 3.}\label{IN3}
\centering
\begin{tabular}{c|cccc|cccc}\hline\hline
 $$ & \multicolumn{4}{c|}{Inexact Newton method} & \multicolumn{4}{c}{QCA} \\ \hline
$(m,n)$   & Iter& Time    & Res   & Ls-iter      & Iter & Time & Res  & Ls-iter \\ \hline
(4,10)  &1.0  &0.00012  &9.2E-15  &1.0          &9.0   &0.00059  &6.4E-12  &1.0\\ 				
(4,50)  &1.0  &0.00947  &2.0E-15  &1.0          &11.0  &0.06154  &4.2E-14  &1.0\\ 				
(4,100) &1.0  &0.14564  &2.1E-15  &1.0          &12.0  &1.02232  &1.9E-15  &1.0\\ 				
\hline\hline				

\end{tabular}

\end{table}
}

{
\begin{table}
\caption{\small   Comparison between Inexact Newton method and QCA method on Problem 4.}\label{IN4}
\centering
\begin{tabular}{c|cccc|cccc}\hline\hline
 $$ & \multicolumn{4}{c|}{Inexact Newton method} & \multicolumn{4}{c}{QCA} \\ \hline
$(m,n)$   & Iter& Time    & Res   & Ls-iter      & Iter & Time & Res  & Ls-iter \\ \hline
(3,10)  &6.7   &0.00020  &8.8E-12  &0 &7.3  &0.00037  &9.6E-12  &0\\ 				
(3,100) &10.3  &0.00934  &7.9E-12  &0 &11.3  &0.01274  &1.1E-11  &0\\ 				
(3,300) &11.6  &0.31909  &1.2E-11  &0 &12.7  &0.35600  &1.3E-11  &0\\ 				
(3,500) &12.4  &1.50356  &7.9E-12  &0 &13.7  &1.67812  &8.5E-12  &0\\ 				
\hline 				
(4,10)  &6.8  &0.00031  &3.7E-12  &0 &7.2  &0.00047  &8.9E-12  &0\\ 				
(4,50)  &8.9  &0.04571  &1.4E-11  &0 &9.4  &0.04826  &1.1E-11  &0\\ 				
(4,100) &9.6  &0.74759  &1.4E-11  &0 &10.3  &0.79482  &1.4E-11  &0\\ 				
\hline 				
(5,10)  &6.6  &0.00047  &5.4E-12  &0  &6.8  &0.00063  &1.0E-11  &0\\ 				
(5,30)  &7.7  &0.15334  &1.5E-11  &0 &8.2 &0.16067  &1.5E-11  &0\\ 				
 \hline\hline				

\end{tabular}

\end{table}
}

{
\begin{table}
\caption{\small   Comparison between Inexact Newton method and QCA method on Problem 5.}\label{IN5}
\centering
\begin{tabular}{c|cccc|cccc}\hline\hline
 $$ & \multicolumn{4}{c|}{Inexact Newton method} & \multicolumn{4}{c}{QCA} \\ \hline
$(m,n)$   & Iter& Time    & Res   & Ls-iter      & Iter & Time & Res  & Ls-iter \\ \hline
(3,10)   &7.9  &0.00016  &7.2E-12   &0.4 &8.8   &0.00035  &4.7E-12  &0.6\\ 				
(3,100)  &10.3  &0.00758  &1.2E-11  &0.1 &11.4  &0.01023  &9.5E-12  &0.6\\ 				
(3,300)  &12.1  &0.27135  &1.2E-11  &0   &13.5  &0.31054  &8.8E-12  &0.5\\ 				
(3,500)  &12.5  &1.44273  &1.4E-11  &0   &14.0  &1.63077  &1.7E-11  &0.4\\ 				
\hline 				
(4,10)   &8.0  &0.00026  &4.2E-12  &0.5 &8.4  &0.00044  &1.1E-11  &0.7\\ 				
(4,50)   &9.7  &0.04921  &8.8E-12  &0.2 &10.6  &0.05329  &8.2E-12  &0.5\\ 				
(4,100)  &10.6  &0.82689  &7.0E-12  &0.2 &11.4  &0.87036  &1.5E-11  &0.7\\ 				
\hline 				
(5,10)  &7.7  &0.00047  &6.3E-12  &0.5 &8.1  &0.00063  &1.0E-11  &0.7\\ 				
(5,30)  &8.6  &0.17388  &6.0E-12  &0.4 &9.0  &0.17493  &1.4E-11  &0.6\\ \hline\hline				

\end{tabular}

\end{table}
}

{
\begin{table}
\caption{\small   Comparison between Regularized Newton method and QCA method on Problem 1.}\label{EN1}
\centering
\begin{tabular}{c|cccc|cccc}\hline\hline
 $$ & \multicolumn{4}{c|}{Regularized Newton method} & \multicolumn{4}{c}{QCA} \\ \hline
$(m,n)$   & Iter & Time    & Res      & Ls-iter      & Iter  & Time    & Res     & Ls-iter \\ \hline
(3,10)   &7.3    &0.00036  &4.3E-12  &0.0           &7.9    &0.00045  &7.1E-12 &0.0\\ 																		(3,100)  &4.6    &0.00734  &1.2E-11  &0.6           &7.7    &0.00933  &7.5E-12 &4.8\\ 																		(3,300)  &5.3    &0.19312  &9.8E-12  &0.8           &7.4    &0.21646  &1.1E-11 &15.2\\ 																		(3,500)  &6.0    &0.97474  &1.4E-11  &0.9           &8.9    &1.18350  &2.1E-11 &18.0\\ 																		\hline 																		
(4,10)   &8.2    &0.00055  &7.4E-12  &0.0           &8.8    &0.00059  &7.9E-12 &0.0\\ 																		(4,50)   &4.3    &0.02707  &8.4E-12  &0.6           &6.4    &0.03507  &1.1E-11 &3.6\\ 																		(4,100)  &5.0    &0.47484  &1.3E-11  &0.8           &8.4    &0.66361  &3.1E-11 &14.3\\ 																		\hline 																		
(5,10)   &8.9    &0.00079  &8.3E-12  &0.0           &9.3    &0.00091  &8.8E-12 &0.0\\ 																		(5,30)   &4.0    &0.11252  &1.5E-12  &1.0           &5.1    &0.11508  &1.8E-11 &1.5\\ 																		\hline\hline\end{tabular}

\end{table}
}

{
\begin{table}
\caption{\small  Comparison between Regularized Newton method and QCA method on Problem 2.}\label{EN2}
\centering
\begin{tabular}{c|cccc|cccc}\hline\hline
 $$ & \multicolumn{4}{c|}{Regularized Newton method} & \multicolumn{4}{c}{QCA} \\ \hline
$(m,n)$  &Iter   &Time     & Res      & Ls-iter      & Iter  & Time    & Res      & Ls-iter \\ \hline\hline
(3,10)   &5.2    &0.00024  &6.7E-12  &0.8           &6.2    &0.00033  &6.8E-12  &3.6\\																		(3,100)  &6.7    &0.00894  &7.9E-12  &0.8           &10.9   &0.01227  &1.1E-11  &35.4\\ 																		(3,300)  &7.2    &0.25393  &1.0E-11  &1.0           &14.3   &0.41420  &2.2E-12  &69.1\\ 																		(3,500)  &7.9    &1.22181  &1.0E-11  &0.8           &16.0   &2.02200  &1.0E-12  &87.0\\ 																		\hline 																		
(4,10)   &4.9    &0.00035  &1.2E-11  &0.7           &5.5   &0.00040  &1.6E-11 &3.8\\ 																		(4,50)   &6.5    &0.03804  &1.2E-11  &0.8           &11.6   &0.06193  &8.4E-13 &42.3\\ 																		(4,100)  &7.0    &0.61673  &1.2E-11  &0.8           &11.8   &0.96383  &1.3E-12 &50.2\\ 																		\hline 																		
(5,10)   &4.7    &0.00048  &1.3E-11  &0.7           &5.4   &0.00051  &7.0E-12 &3.5\\ 																		(5,30)   &6.0    &0.13580  &1.3E-11  &0.7           &10.1   &0.20807  &1.7E-12 &31.6\\
\hline\hline			
\end{tabular}

\end{table}

{
\begin{table}
\caption{\small Comparison between Regularized Newton method and QCA method on Problem 3.}\label{EN3}
\centering
\begin{tabular}{c|cccc|cccc}\hline\hline
 $$ & \multicolumn{4}{c|}{Regularized Newton method} & \multicolumn{4}{c}{QCA} \\ \hline
$(m,n)$   & Iter  & Time    & Res      & Ls-iter      & Iter & Time    & Res     & Ls-iter \\ \hline														(4,10)    &15.0    &0.00101  &6.9E-16  &0           &18.0   &0.00132  &1.6E-12 &0\\ 																		(4,50)    &16.0    &0.09657  &3.7E-11  &0           &20.0   &0.09892  &3.9E-12 &0\\ 																		(4,100)   &17.0    &1.58356  &9.0E-12  &0           &21.0   &1.61399  &8.8E-14 &0\\ 																		\hline\hline 																		

\end{tabular}

\end{table}

{
\begin{table}
\caption{\small  Comparison between Regularized Newton method and QCA method on Problem 4.}\label{EN4}
\centering
\begin{tabular}{c|cccc|cccc}\hline\hline
 $$ & \multicolumn{4}{c|}{Regularized Newton method} & \multicolumn{4}{c}{QCA} \\ \hline
$(m,n)$   & Iter  & Time    & Res      & Ls-iter      & Iter & Time    & Res     & Ls-iter \\ \hline
(3,10)    &6.6    &0.00029  &7.4E-12  &0.0           &7.0   &0.00036  &6.7E-12 &0.0\\ 																		(3,100)   &4.8    &0.00669  &1.4E-11  &0.6           &7.3   &0.00823  &4.6E-12 &5.8\\ 																		(3,300)   &5.4    &0.18935  &1.2E-11  &0.5           &9.3   &0.28820  &5.9E-11 &17.9\\ 																		(3,500)   &5.6    &0.91665  &9.6E-12  &0.8           &9.2   &1.19580  &2.0E-11 &20.0\\ 																		\hline 																		
(4,10)    &7.8    &0.00048  &8.1E-12  &0.0           &8.2   &0.00052  &6.9E-12 &0.0\\ 																		(4,50)    &4.4    &0.02734  &9.8E-12  &0.6           &6.5   &0.03502  &1.0E-11 &4.0\\ 																		(4,100)   &5.0    &0.46779  &1.4E-11  &0.7           &9.3   &0.72021  &3.7E-11 &17.0\\ 																		\hline 																		
(5,10)    &8.5    &0.00075  &8.8E-12  &0.0           &8.9   &0.00077  &1.0E-11 &0.0\\ 																		(5,30)    &3.9    &0.10700  &1.6E-11  &1.2           &5.1   &0.10740  &1.2E-11 &1.7\\ 																		\hline\hline
\end{tabular}

\end{table}

{
\begin{table}
\caption{\small  Comparison between Regularized Newton method and QCA method on Problem 5.}\label{EN5}
\centering
\begin{tabular}{c|cccc|cccc}\hline\hline
 $$ & \multicolumn{4}{c|}{Regularized Newton method} & \multicolumn{4}{c}{QCA} \\ \hline
$(m,n)$   & Iter   & Time    & Res      & Ls-iter      & Iter & Time & Res  & Ls-iter \\ \hline
(3,10)    &8.0     &0.00031  &2.0E-12  &0.5            &10.0   &0.00043  &6.9E-12 &12.8\\ 																		(3,100)   &11.9    &0.01314  &5.5E-12  &0.8            &14.7   &0.01385  &5.3E-11 &73.4\\ 																		(3,300)   &14.2    &0.36332  &2.1E-11  &0.7            &17.5   &0.40539  &1.0E-11 &109.4\\ 																		(3,500)   &15.3    &2.02149  &2.0E-11  &0.5            &18.8   &2.22143  &4.6E-12 &127.9\\ 																		\hline 																		
(4,10)    &8.6     &0.00053  &4.1E-12  &0.5            &11.4   &0.00061  &1.3E-11 &26.7\\ 																		(4,50)    &12.5    &0.06472  &1.6E-11  &0.8            &16.1   &0.08161  &3.1E-12 &95.1\\ 																		(4,100)   &14.9    &1.16723  &1.6E-11  &0.8            &19.5   &1.50176  &5.2E-13 &144.7\\ 																		\hline 																		
(5,10)    &9.2     &0.00080  &6.9E-12  &0.5            &12.7   &0.00097  &6.3E-12 &44.3\\ 																		(5,30)    &12.8    &0.27238  &1.4E-11  &1.8            &17.3   &0.33700  &3.9E-13 &115.6\\ 																		\hline\hline
\end{tabular}

\end{table}

From the data in the Tables \ref{IN1}, \ref{IN2}, \ref{IN3}, \ref{IN4}, \ref{IN5}, \ref{EN1}, \ref{EN2},
\ref{EN3}, \ref{EN4} and \ref{EN5}, we can see that the proposed methods are effective for all test problems.
In terms of the number of iterations and CPU time, Inexact Newton method and Regularized Newton method are better than QCA method, and the number of linear search of the Regularized Newton method are far less than that of the QCA method.

\end{document}